\documentclass[10pt]{amsart}

	\usepackage[pagebackref,  pdfpagemode=FullScreen,  colorlinks=true]{hyperref}	   
	
	\usepackage{graphicx, subfig, color, verbatim, calc}

	\newlength{\len}		
	\newlength{\martot}		
	\newlength{\mar}		
	\newtheorem{thm}{Theorem}
  	\newtheorem{lem}{Lemma}
  	\newtheorem{prop}{Proposition}

	\theoremstyle{definition}
  	
	\newtheorem*{ex}{Example}

	\theoremstyle{remark}
  	\newtheorem{rem}{Remark}
	
	\newcommand{\MM}{\mathcal M}
	\newcommand{\BM}{\overline{\mathcal M}}

	\makeatletter
	\@namedef{subjclassname@2010}{%
	\textup{2010} Mathematics Subject Classification}
	\makeatother

\begin{document}

\title[Extremality of loci of hyperelliptic curves with marked Weierstrass points]{Extremality of loci of hyperelliptic curves\\ with marked Weierstrass points}

\author{Dawei Chen}
\address{Department of Mathematics, Boston College, Chestnut Hill, MA 02467}
\email{dawei.chen@bc.edu}

\author{Nicola Tarasca}
\address{Department of Mathematics, University of Utah, Salt Lake City, UT 84112}
\email{tarasca@math.utah.edu}


\thanks{During the preparation of this article the first author was partially supported by NSF CAREER grant DMS-1350396.}

\begin{abstract}
The locus of genus-two curves with $n$ marked Weierstrass points has codimension $n$ inside the moduli space of genus-two curves with $n$~marked points, for $n\leq 6$. It is well known that the class of the closure of the divisor obtained for $n=1$ spans an extremal ray of the cone of effective divisor classes. 
We generalize this result for all $n$: we show that the class of the closure of the locus of genus-two curves with $n$ marked Weierstrass points spans an extremal ray of the cone of effective classes of codimension $n$, for $n\leq 6$. A related construction produces extremal nef curve classes in  moduli spaces of pointed elliptic curves.
\end{abstract}

\maketitle

Every smooth curve of genus two has a unique map of degree two to the projective line, ramified at six  points, called {\it Weierstrass points}. It follows that
the locus $\mathcal{H}yp_{2,n}$ of curves of genus two with $n$~marked Weierstrass points  has co\-di\-men\-sion $n$ inside the moduli space $\MM_{2,n}$
of smooth curves of genus two with $n$ marked points, for $1\leq n \leq 6$. 
In this paper, we study the classes of the closures of the loci $\mathcal{H}yp_{2,n}$ inside the moduli space of stable curves $\BM_{2,n}$.

The cone of effective codimension-one classes on $\BM_{2,1}$ is explicitly described in \cite{MR2701950} and \cite{MR2216265}, and encodes the rational contractions of $\BM_{2,1}$. 
It is thus natural to study cones of effective classes of higher codimension.
The following is one of the first results in this direction.

\begin{thm}
\label{Hyp2nres}
For $1\leq n \leq 6$, the class of $\overline{\mathcal{H}yp}_{2,n}$ is rigid and extremal in the cone of effective classes of codimension~$n$ in 
$\BM_{2,n}$. 
\end{thm}

Theorem \ref{Hyp2nres} motivates the computation of the classes of the loci $\overline{\mathcal{H}yp}_{2,n}$. The class of the divisor $\overline{\mathcal{H}yp}_{2,1}$ has been computed in \cite{MR910206}, and the class of the codimension-two locus $\overline{\mathcal{H}yp}_{2,2}$ has been computed in \cite{DR}. In \S \ref{Hyp23} we study the next nontrivial case.

\begin{thm}
\label{hyp23class}
In $A^3(\BM_{2,3})$, we have
\begin{eqnarray*}
\left[\overline{\mathcal{H}yp}_{2,3}\right] &=& \left( (3\omega_1-\lambda-\delta_1)\cdot(3\omega_2-\lambda-\delta_1) -(\delta_{0:\{1,2\}}+\delta_{0:3})\cdot(3\omega_1-\lambda-\delta_1)-\gamma_{1:\emptyset} - \gamma_{1:\{3\}} \right)\\
&&\cdot (3\omega_3-\lambda-\delta_1 - \delta_{0:\{1,3\}} - \delta_{0:\{2,3\}}) - \gamma_{1:\{1\}}\cdot(2\psi_1-\delta_{1:\{1\}})- \gamma_{1:\{2\}}\cdot(2\psi_2-\delta_{1:\{2\}})\\
&&{} - \gamma_{1:\emptyset} \cdot(\psi_1-\delta_{0:\{1,3\}}) .
\end{eqnarray*}
\end{thm}

For elliptic curves, the difference of two ramification points of a degree-$2$ map to the projective line can be regarded as a $2$-torsion point. 
In a somewhat similar fashion, we consider in general the locus of points on elliptic curves whose pairwise differences are $k$-torsion points.
More precisely, for $k\geq 2$ and $2\leq n \leq k^2$, consider the following one-dimensional locus in $\MM_{1,n}$
\[
\mathcal{T}\!or^k_{1,n} := \{ [C, p_1, \dots, p_n] \in \MM_{1,n} \,|\, k p_1\sim \cdots \sim k p_n \}.
\]
Note that ${\mathcal{T}\!or}^k_{1,n}$ might be reducible: $\mathcal{T}\!or^d_{1,n}$ is a sub-curve of $\mathcal{T}\!or^k_{1,n}$ for all divisors $d$ of $k$. 

The class of the divisor $\overline{\mathcal{T}\!or}^2_{1,2}$ is in the interior of the two-dimensional cone of effective divisor classes in $\BM_{1,2}$ and spans an extremal ray of the cone of nef divisor classes in $\BM_{1,2}$ (\cite{MR2701950}). 

\begin{thm}
\label{Tor}
For $k\geq 2$ and $2\leq n \leq k^2$, the class of $\overline{\mathcal{T}\!or}^k_{1,n}$ spans an extremal ray of the cone of nef curve classes in $\BM_{1,n}$, and this ray does not dependent on $k$. 
\end{thm}

\vskip4pt

\noindent {\bf Structure of the paper.} The proof of Theorem \ref{Tor} is in \S \ref{g=1} --- this section is independent from the rest of the paper. In \S \ref{Hypgnrt} we collect some classical results on classes of hyperelliptic loci which are needed later on.  
The proof of Theorem \ref{Hyp2nres} in the case $n=2$ is in \S \ref{extrHyp22} and is based on the explicit description of the codimension-two class $[\overline{\mathcal{H}yp}_{2,2}]$ presented in \S \ref{genus2classes}. 
In \S \ref{recS} we prove a recursive argument that works in a more general context, and we thus  complete the proof of Theorem \ref{Hyp2nres}.
Finally, we prove Theorem \ref{hyp23class} in \S \ref{Hyp23} using the description of the classes $[\overline{\mathcal{H}yp}_{2,1}]$ and $[\overline{\mathcal{H}yp}_{2,2}]$ from \S \ref{genus2classes}.

\vskip4pt

\noindent {\bf Notation.} 
We use throughout the following notation for divisor classes on $\BM_{g,n}$. The class $\psi_i$ is the cotangent class at the marked point $i$, and the class $\omega_i$ is the pull-back of the class $\psi_i$ via the map $\rho_i\colon \BM_{g,n}\rightarrow \BM_{g,1}$ obtained by forgetting all marked points but the point $i$. The class $\delta_{\rm irr}$ is the class of the closure of the locus of nodal irreducible curves. We denote by $\lambda$ the pull-back of the first Chern class of the Hodge bundle over $\BM_g$. For $i\in\{0,\dots,g\}$ and $J\subseteq\{1,\dots,n\}$, we denote by $\delta_{i:J}$ the class of the divisor $\Delta_{i:J}$ whose general element has a component of genus $i$ containing the points marked by indices in $J$ and meeting a component of genus $g-i$ containing the remaining marked points. One has $\delta_{i:J}=\delta_{g-i:J^c}$. We will denote by $\delta_{i:j}$ the sum of  all distinct divisor classes $\delta_{i:J}$ such that $|J|=j$, and by $\delta_i$ the sum of all distinct classes $\delta_{i:j}$ for all possible values of $j$.
Let $\pi_k\colon \BM_{g,n}\rightarrow \BM_{g,n-1}$ be the map obtained by forgetting the $k$-th marked point. Note that $\pi_k^*(\delta_i)=\delta_i$ on $\BM_{g,n}$, for $n\geq 2$.

\vskip4pt

We also use the following codimension-two tautological classes on $\BM_{g,n}$.
For $J\subseteq\{1,\dots,n\}$, let $\gamma_{1:J}$ be the class of the locus $\Gamma_{1:J}$ of curves whose general element has an elliptic component  containing exactly the points marked by indices in $J$, and meeting in two points a component of genus $g-2$ containing the remaining marked points. 

\vskip4pt

Throughout we work over an algebraically closed field of characteristic $0$. 
All cycle classes are stack fundamental classes, 
and all cohomology and Chow groups are taken with rational coefficients. 
We implicitly assume real coefficients when we consider nef classes and closures of cones of effective classes.

\vskip4pt

\noindent {\bf Acknowledgements.} We would like to thank Izzet Coskun for helpful discussions on related ideas,  Dan Petersen for pointing our attention to the results in \cite{MR3179574} and \cite{Petersen}, and the referee for suggesting many improvements on the exposition.

\section{Extremal nef curve classes on \texorpdfstring{$\BM_{1,n}$}{BM(1,n)}}
\label{g=1}

In this section we show that the class of the one-dimensional locus $\overline{\mathcal{T}\!or}^k_{1,n}$ spans an extremal ray of the cone ${\rm Nef}^{n-1}(\BM_{1,n})$  of nef curve classes on $\BM_{1,n}$, for $k\geq 2$ and $2\leq n \leq k^2$.

By definition, the cone ${\rm Nef}^{n-1}(\BM_{1,n})$  is dual to the cone of pseudo-effective divisors $\overline{\rm Eff}^1(\BM_{1,n})$.
A subcone $S$ of $\overline{\rm Eff}^1(\BM_{1,n})$ is {\it extremal in} $\overline{\rm Eff}^1(\BM_{1,n})$ if whenever $E_1+E_2\in S$, then $E_1, E_2 \in S$.
We first show that the cone generated by the boundary divisor classes is extremal in $\overline{\rm Eff}^1(\BM_{1,n})$. 

\begin{lem}
\label{Lg1}
The cone generated by the classes $\delta_{0:J}$ for $J\subseteq \{1,\dots, n\}$ is extremal in $\overline{\rm Eff}^1(\BM_{1,n})$ for~$n\geq 2$.
\end{lem}

\begin{proof}
We follow the strategy of \cite[Corollary 1.4.7]{MR2701950}, where Rulla shows that the cone generated by the boundary divisor classes on $\BM_g$ is extremal in $\overline{\rm Eff}^1(\BM_{g})$.

We first show that the classes $\delta_{0:\{i,j\}}$ are extremal in $\overline{\rm Eff}^1(\BM_{1,n})$, for $n\geq 3$ and $i,j\in \{1,\dots,n\}$. 
Consider the one-dimensional family of curves $C_{\{i,j\}}$ obtained by attaching a rational component containing the points with markings $i$ and $j$ at a moving point of an elliptic curve containing the remaining $n-2$ marked points. The curve $C_{\{i,j\}}$ is a moving curve in $\Delta_{0:\{i,j\}}$, and one has $C_{\{i,j\}}\cdot \delta_{0:\{i,j\}}<0$. It follows that $\delta_{0:\{i,j\}}$ is rigid and extremal in $\overline{\rm Eff}^1(\BM_{1,n})$.
Moreover, $C_{\{i,j\}}$ has empty intersection with $\delta_{0:J}$ for $|J|=2$ and $J\not= \{i,j\}$. From \cite[Lemma 1.4.6]{MR2701950}, the cone generated by the classes $\delta_{0:J}$ for $|J|=2$ is extremal in $\overline{\rm Eff}^1(\BM_{1,n})$.

We then use the following recursion on $k$, for $3\leq k \leq n-1$. Suppose that the cone generated by all classes $\delta_{0:J}$ with $|J| < k<n$ is extremal in $\overline{\rm Eff}^1(\BM_{1,n})$. For $i_1,\dots,i_k\in \{1,\dots,n\}$, consider the one-dimensional family $C_{\{i_1,\dots, i_k\}}$ obtained by attaching a rational component containing the points with markings $i_1,\dots, i_k$ at a moving point of an elliptic curve containing the remaining $n-k$ marked points.
One has $C_{\{i_1,\dots, i_k\}} \cdot \delta_{0:\{i_1,\dots, i_k\}}<0$ and $C_{\{i_1,\dots, i_k\}} \cdot \delta_{0:J}=0$ for $|J|\leq k$ and $J\not=\{i_1,\dots, i_k\}$. Again from \cite[Lemma 1.4.6]{MR2701950}, the cone generated by the classes $\delta_{0:J}$ for $|J|\leq k$ is thus extremal in $\overline{\rm Eff}^1(\BM_{1,n})$.

Finally, let us consider the class $\delta_{0:\{1,\dots,n\}}$ for $n\geq 2$. Consider the one-dimensional family $E$ obtained by attaching a rational curve containing all marked points to a base point of a pencil of plane cubics. One has $E\cdot \delta_{0:\{1,\dots,n\}} <0$ and $E\cdot \delta_{0:J}=0$ for $|J|<n$. Hence, the cone generated by $\delta_{0:\{1,\dots,n\}}$ and $\delta_{0:J}$ with $|J|<n$ is extremal in $\overline{\rm Eff}^1(\BM_{1,n})$.
\end{proof}

We are now ready to prove Theorem \ref{Tor}. 

\begin{proof}[Proof of Theorem \ref{Tor}]
Singular elements in $\overline{\mathcal{T}\!or}^k_{1,n}$ do not have rational tails. 
Indeed, consider a singular pointed curve $[C,p_1,\dots, p_n]$ inside the closure of ${\mathcal{T}\!or}^k_{1,n}$.
The condition $k p_i \sim k p_j$ means that there exists an admissible cover $\pi\colon C\rightarrow \mathbb{P}^1$ of degree $k$ totally ramified at $p_i$ and $p_j$. Suppose $C$ has a rational tail $R$ containing $p_i$ and $p_j$. By the Riemann-Hurwitz formula, $R$ does not contain any other ramification point of $\pi$. Since $\overline{C \setminus R}$ has arithmetic genus $1$, one has 
$\deg(\pi|_{\overline{C \setminus R}})>1$. Hence, the tail $R$ has to meet the other components of $C$ in more than one point, a contradiction.
It follows that  $[\overline{\mathcal{T}\!or}^k_{1,n}]$ has zero intersection with all divisor classes $\delta_{0:J}$. 
Note that by the projection formula relative to the natural map $\BM_{1,n}\rightarrow \BM_{1,1}$, the class $[\overline{\mathcal{T}\!or}^k_{1,n}]$ has positive intersection with the divisor class $\lambda$. 

From Lemma \ref{Lg1}, the cone generated by the classes $\delta_{0:J}$ is extremal in $\overline{\rm Eff}^1(\BM_{1,n})$, hence by duality the class $[\overline{\mathcal{T}\!or}^k_{1,n}]$ spans an extremal ray of ${\rm Nef}^{n-1}(\BM_{1,n})$, and this ray does not depend on $k$.  
\end{proof}

\section{On hyperelliptic loci}
\label{Hypgnrt}

In the following, we collect some well-known facts about classes of hyperelliptic loci which we will use later. 
For $g\geq 2$ and $0\leq n \leq 2g+2$, let
\[
\mathcal{H}yp_{g,n} := \left\{ [C, p_1, \dots, p_n] \in \MM_{g,n} \,|\, C \,\, \mbox{is hyperelliptic and $h^0\left(C,\mathcal{O}_C(2p_i)\right)\geq 2$, for $i=1,\dots, n$}\right\}
\]
be the locus of hyperelliptic curves of genus $g$ with $n$ marked Weierstrass points. The locus $\mathcal{H}yp_{g,n}$ has codimension $g-2+n$ in  the moduli space $\MM_{g,n}$ of smooth curves of genus $g$ with $n$ marked points. 
The class of the closure $\overline{\mathcal{H}yp}_{g,n}$ is tautological on the moduli of stable curves $\BM_{g,n}$ (\cite{MR2120989}).

Let $\MM_{g,n}^{rt}$ be the moduli space of curves with rational tails.
From \cite{MR2120989} or \cite{MR2176546}, the tautological group $R^{g-2+n}(\MM^{rt}_{g,n})$ is one-dimensional. When $n=0$, $R^{g-2}(\MM_{g})$ is one-dimensional and is generated by the class of the hyperelliptic locus ${\mathcal{H}yp}_g$, or equivalently the class $\kappa_{g-2}$ (\cite{MR1346214} and \cite{MR1722541}). Let ${\mathcal{H}yp}^{rt}_{g,n}$ be the restriction of $\overline{\mathcal{H}yp}_{g,n}$ to $\MM^{rt}_{g,n}$. Since the push-forward of $[{\mathcal{H}yp}^{rt}_{g,n}]$ via the natural map $\MM^{rt}_{g,n}\rightarrow \MM_g$ is a positive multiple of $[{\mathcal{H}yp}_g]$, it follows that $[{\mathcal{H}yp}^{rt}_{g,n}]$ is non-zero and generates $R^{g-2+n}(\MM^{rt}_{g,n})$. 

Equivalently, $R^{g-2+n}(\MM^{rt}_{g,n})$ is generated by the decorated class $\delta_{g,\psi^{g-1}}$ defined in the following way.
Consider the glueing map 
\[
\xi\colon \BM_{g,1} \rightarrow \BM_{g,n}
\]
obtained by attaching a chain of $n-1$ rational components at the marked point of an element in $\BM_{g,1}$. We fix the markings in an increasing order, from the inner rational component to the outer one. From the rational equivalence of points in $\BM_{0,n+1}$, the classes of the loci obtained by permuting the markings on the image of $\xi$ are all rationally equivalent. The class $\delta_{g,\psi^{g-1}}$ is defined as the push-forward of the class $\psi^{g-1}$ in $R^{g-1}(\BM_{g,1})$ via the map $\xi$.  

Let $\pi\colon \BM_{g,1}\rightarrow \BM_g$ and $\pi_n\colon \BM_{g,n}\rightarrow \BM_{g,n-1}$ be the natural maps. Note that $(\pi_n)_* \delta_{g,\psi^{g-1}} = \delta_{g,\psi^{g-1}}$ in $R^{g-2+n-1}(\MM_{g,n-1})$ for $n\geq 3$, and $(\pi_2)_* \delta_{g,\psi^{g-1}} = \psi^{g-1}$. 
Since $\kappa_{g-2}:=\pi_*(\psi^{g-1})$ is non-zero in $R^{g-2}(\MM_g)$, we conclude $\delta_{g,\psi^{g-1}}$  is non-zero in $R^{g-2+n}(\MM^{rt}_{g,n})$. 

\begin{ex}
In the case $g=2$, for $2\leq n \leq 6$ we have
\[
\left[\mathcal{H}yp^{rt}_{2,n} \right] = \frac{6!}{2\cdot (6-n)!}\, \delta_{2,\psi} \in R^n(\MM_{2,n}^{rt}).
\]
Indeed, let us write $\left[{\mathcal{H}yp}^{rt}_{2,n}\right] = \alpha \, \delta_{2,\psi}$ in $R^n(\MM_{2,n}^{rt})$. In order to determine the coefficient $\alpha$, we intersect both sides of the equation with a test space. Let $C$ be a smooth curve of genus $2$, and let $C[n]$ be the $n$-th Fulton-MacPherson compactification of the space of $n$ distinct points of $C$ (\cite{MR1259368}).  The natural map $C[n+1]\rightarrow C[n]$ gives an $n$-dimensional family of genus-$2$ curves with rational tails. 
Weierstrass points on $C$ are ramification points of the hyperelliptic double covering.
Analyzing the Hurwitz space of admissible double coverings, it is easy to see that the intersection $[{\mathcal{H}yp}^{rt}_{2,n}] \cdot C[n]$ corresponds to all ordered $n$-tuples of Weierstrass points in $C$,
and is transversal. We deduce that $[{\mathcal{H}yp}^{rt}_{2,n}] \cdot C[n] = 6!/(6-n)!$. On the other hand, one has $\delta_{2,\psi} \cdot C[n] = 
\psi \cdot \xi^*(C[n]) = \psi \cdot C[1]  = 2$, whence the statement. 
\end{ex}

\section{A recursive argument}
\label{recS}

Let $N^k(\BM_{g,n})$ be the group of co\-di\-men\-sion-$k$ cycles on $\BM_{g,n}$ modulo numerical equivalence.
We denote by ${\rm Eff}^k(\BM_{g,n})\subset N^k(\BM_{g,n})$ the {\it cone  of effective cycle classes}, and by ${\rm REff}^k(\BM_{g,n})\subseteq {\rm Eff}^k(\BM_{g,n})$ the {\it sub-cone  of effective tautological classes} (see \cite{MR2120989} for tautological classes on $\BM_{g,n}$).

A cycle class $E$ inside a cone $K\subset N^k(\BM_{g,n})$ is called {\it extremal in} $K$ if whenever two cycle classes $E_1$ and $E_2$ in $K$ are such that $E= E_1+E_2$,  then both $E_1$ and $E_2$ lie in the ray spanned by $E$. An effective cycle class $E$ is called {\it rigid} if any effective cycle with class $m E$  is supported on the support of $E$.

\begin{thm}
\label{rec}
Given $g\geq 2$,
if $[\overline{\mathcal{H}yp}_{g,2}]$ is rigid and extremal in ${\rm REff}^g(\BM_{g,2})$,
then $[\overline{\mathcal{H}yp}_{g,n}]$ is rigid and extremal in ${\rm REff}^{g-2+n}(\BM_{g,n})$, for $3\leq n \leq 2g+2$. 
\end{thm}

\begin{proof}
Let $n\geq 3$ and assume that the statement is true for $\overline{\mathcal{H}yp}_{g,n-1}$.
Suppose that 
\begin{eqnarray}
\label{sumaiXi}
\left[\overline{\mathcal{H}yp}_{g,n}\right] = \sum_i a_i [X_i],
\end{eqnarray}
with $a_i>0$, $X_i$ irreducible, tautological, effective of codimension $n$, and $[X_i]$ not proportional to $[\overline{\mathcal{H}yp}_{g,n}]$, for all $i$.

Since $R^{g-2+n}(\MM^{rt}_{g,n})$ is generated by $\delta_{g,\psi^{g-1}}$ (see \S \ref{Hypgnrt}), we can express the class of each $X_i$ as 
\[
[X_i]= c_i \, \delta_{g,\psi^{g-1}} + B_i, 
\]
where $c_i$ is a non-negative coefficient, and $B_i$ is a (not necessarily effective) cycle class in $R^{g-2+n}(\BM_{g,n})$ with $B_i =0$ in $R^{g-2+n}(\MM^{rt}_{g,n})$. 
Let $\pi_j \colon \BM_{g,n} \rightarrow \BM_{g,n-1}$ be the map obtained by forgetting the point~$j$. Applying $(\pi_j)_*$ to (\ref{sumaiXi}), we have
\begin{eqnarray}
\label{sumaipiXi}
(2g+2-(n-1))  \left[\overline{\mathcal{H}yp}_{g,n-1} \right] = \sum_i a_i (\pi_j)_*[X_i].
\end{eqnarray}

Pick a locus $X_i$ appearing on the right side of (\ref{sumaiXi}). Consider two cases. 
First, suppose $(\pi_1)_* [X_i] = \cdots = (\pi_n)_* [X_i] =0$. Note that 
$
(\pi_j)_* \delta_{g,\psi^{g-1}} = \delta_{g,\psi^{g-1}}
$
in $R^{g-2+n-1}(\MM^{rt}_{g,n-1})$, for all $j=1,\dots,n$. Since~$B_i =0$ in $R^{g-2+n}(\MM^{rt}_{g,n})$, using the exact sequence 
\[
A^{g-2+n}(\BM_{g,n}\setminus \MM^{rt}_{g,n}) \rightarrow A^{g-2+n}(\BM_{g,n}) \rightarrow A^{g-2+n}(\MM^{rt}_{g,n}) \rightarrow 0
\]  
we can assume that $B_i$ is represented by a linear combination of cycle classes supported in $\BM_{g,n}\setminus \MM^{rt}_{g,n}$. 
An element in the support of such a cycle does not have an irreducible and smooth component of genus~$g$, hence
$(\pi_j)_* B_i =0$  in $A^{g-2+n-1}(\MM^{rt}_{g,n-1})$, for all $j=1,\dots,n$. We deduce that $c_i=0$, that is, $[X_i]=0$ in $R^{g-2+n}(\MM^{rt}_{g,n})$.

For the other case, suppose $(\pi_1)_*[X_i]$ is nonzero. Since the class $[\overline{\mathcal{H}yp}_{g,n-1}]$ is rigid and extremal in ${\rm REff}^{g+n-3}(\BM_{g,n-1})$, from (\ref{sumaipiXi}) we deduce that $(\pi_1)_*[X_i]$ is a positive multiple of the class of $\overline{\mathcal{H}yp}_{g,n-1}$ and moreover $X_i\subset (\pi_1)^{-1}\overline{\mathcal{H}yp}_{g,n-1}$. This implies that $(\pi_2)_*[X_i], \dots, (\pi_{n})_*[X_i]$ are also nonzero. It follows that $X_i$ is in the intersection of all the $(\pi_j)^{-1}\overline{\mathcal{H}yp}_{g,n-1}$, for $j=1,\dots,n$. In particular, any $n-1$  marked points in a general element of $X_i$ are distinct Weierstrass points. 
Hence, all $n$ marked points must be distinct Weierstrass points. (Note that $n\geq 3$.)
This forces $[X_i]$ to be a positive multiple of $[\overline{\mathcal{H}yp}_{g,n}]$, a contradiction.

Finally, the above steps show that $[X_i]=0$ in $R^{g-2+n}(\MM^{rt}_{g,n})$, for all $i$. This yields a contradiction since $\left[\overline{\mathcal{H}yp}_{g,n}\right]\not=0$ in $R^{g-2+n}(\MM^{rt}_{g,n})$ (see \S \ref{Hypgnrt}), hence $[\overline{\mathcal{H}yp}_{g,n}]$ is extremal in ${\rm REff}^{g-2+n}(\BM_{g,n})$.

Suppose that $E:= m\, [\overline{\mathcal{H}yp}_{g,n}]$ is effective. Since $(\pi_j)_* (E)= (2g-n+3) m \, [\overline{\mathcal{H}yp}_{g,n-1}]$ and $ [\overline{\mathcal{H}yp}_{g,n-1}]$ is rigid, $(\pi_j)_* (E)$ is supported on $\overline{\mathcal{H}yp}_{g,n-1}$, for $j=1,\dots, n$. This implies that $E$ is supported on the intersection of all the $(\pi_j)^{-1}\overline{\mathcal{H}yp}_{g,n-1}$, for $j=1,\dots,n$. We conclude that $E$ is supported on $\overline{\mathcal{H}yp}_{g,n}$ and $[\overline{\mathcal{H}yp}_{g,n}]$ is rigid.
\end{proof}

\begin{rem}
The classes $[\overline{\mathcal{H}yp}_{3}]$, $[\overline{\mathcal{H}yp}_{3,1}]$, and $[\overline{\mathcal{H}yp}_{4}]$ are known to be extremal, respectively in ${\rm Eff}^1(\BM_{3})$ (\cite{MR2701950}), ${\rm Eff}^2(\BM_{3,1})$, and ${\rm Eff}^2(\BM_{4})$ (\cite{CC}). It is natural to wonder whether $[\overline{\mathcal{H}yp}_{g,n}]$ is extremal in ${\rm REff}^{g-2+n}(\BM_{g,n})$, for all $g\geq 2$ and $0\leq n \leq 2g+2$. By Theorem \ref{rec}, it is enough to study the cases~$n\leq 2$.
\end{rem}

\section{Loci of Weierstrass points on curves of genus \texorpdfstring{$2$}{2}}
\label{genus2}

In this section, we complete the proof of Theorem \ref{Hyp2nres}. It is enough to show that $[\overline{\mathcal{H}yp}_{2,n}]$ is rigid and extremal in ${\rm REff}^n(\BM_{2,n})$, and use the fact that ${\rm REff}^*(\BM_{2,n})= {\rm Eff}^*(\BM_{2,n})$ for small values of~$n$. Indeed, according to \cite{MR3179574} and \cite[Theorem~3.8]{Petersen}, all even cohomology of $\BM_{2,n}$ is tautological for $n<20$. Note that the Betti numbers of $\BM_{2,n}$ for $n\leq 7$ have been computed in \cite{MR1672112} and \cite{MR2538614}.

\subsection{The classes for \texorpdfstring{$n=1,2$}{n=1,2}}
\label{genus2classes}
When $n=1$, the class of the divisor $\overline{\mathcal{H}yp}_{2,1}$ in $\BM_{2,1}$ is
\begin{align}
\label{Hyp21}
\left[\overline{\mathcal{H}yp}_{2,1}\right] = 3 \omega - \frac{1}{10} \delta_{\rm irr} -\frac{6}{5}\delta_{1} = 3 \omega - \lambda -\delta_{1} \in {\rm Pic}(\BM_{2,1})
\end{align}
(\cite[Theorem 2.2]{MR910206}), and $[\overline{\mathcal{H}yp}_{2,1}]$ is rigid and extremal in ${\rm Eff}^1(\BM_{2,1})$ (\cite{MR2701950}).
When $n=2$, the class of the double ramification locus $\overline{\mathcal{H}yp}_{2,2}$ in $\BM_{2,2}$ is
\begin{align}
\label{Hyp22old}
\left[\overline{\mathcal{H}yp}_{2,2} \right]= 6 \psi_1 \cdot \psi_2 - \frac{3}{2} (\psi_1^2 + \psi_2^2) - (\psi_1 + \psi_2) \cdot \left( \frac{21}{10}\delta_{1:1} + \frac{3}{5}\delta_{1:0} + \frac{1}{20} \delta_{\rm irr}  \right) \in A^2(\BM_{2,2})
\end{align}
and $\overline{\mathcal{H}yp}_{2,2}$ is not a complete intersection (\cite{DR}). 
Expressing products of divisor classes in terms of decorated boundary strata classes, we have
\begin{align}
\label{Hyp22dbs}
\left[\overline{\mathcal{H}yp}_{2,2}\right] = 5\delta_{2,w}+9\delta_{11|} +\frac{5}{8} \delta_{01|} -\frac{1}{8} \left(\delta_{01|1} + \delta_{01|2} +\delta_{01|12} \right) +2 \gamma_{1:\emptyset} +\frac{1}{24} \delta_{00} .
\end{align}
Here, $\delta_{2,w}$ is the class of the locus of curves with a rational tail containing both marked points attached at a Weierstrass point on a component of genus $2$; $\delta_{11|}$ is the class of the locus of curves whose general element has two elliptic tails attached at a rational component containing both marked points; 
$\delta_{01|}$, $\delta_{01|1}$, $\delta_{01|2}$, $\delta_{01|12}$ are the classes of the loci of curves whose general element has an elliptic tail attached at a nodal rational component with the points distributed in the following way: for the class $\delta_{01|}$ both marked points are on the rational component, for $\delta_{01|i}$ the point $i$ is on the elliptic component and the other marked point is on the rational component, for $\delta_{01|12}$ both marked points are on the elliptic component;
$\gamma_{1:\emptyset}$ is the class of the locus of curves with an elliptic component meeting in two points a rational component containing both marked points;
finally, $\delta_{00}$ is the class of the locus whose general element is a rational curve with two non-disconnecting nodes.

In \S \ref{extrHyp22}, we will use in a crucial way the expression in (\ref{Hyp22dbs}).
In \S \ref{Hyp23}, we will also use the following description. Let $\pi_i\colon \BM_{2,2} \rightarrow \BM_{2,1}$ be the map obtained by forgetting the point $i$, for $i=1,2$. 

\begin{lem}
\label{pi1*pi2}
The following equality holds in $A^2(\BM_{2,2})$
\[
\pi_1^*\left[ \overline{\mathcal{H}yp}_{2,1}\right]  \cdot \pi_2^*\left[ \overline{\mathcal{H}yp}_{2,1} \right] = \left[\overline{\mathcal{H}yp}_{2,2}\right] + \gamma_{1:\emptyset} + \delta_{2,w}.
\]
In particular,  $\pi_1^{-1}(\overline{\mathcal{H}yp}_{2,1}) \cap \pi_2^{-1}(\overline{\mathcal{H}yp}_{2,1})$ is the union of 
the supports of $[\overline{\mathcal{H}yp}_{2,2}]$, $\gamma_{1:\emptyset}$, and $\delta_{2,w}$. 
\end{lem}

\begin{proof}
The desired equality follows from \eqref{Hyp21} and \eqref{Hyp22old}. Since the supports of $[\overline{\mathcal{H}yp}_{2,2}]$, 
$\gamma_{1:\emptyset}$, and $\delta_{2,w}$ are contained in $\pi_i^{-1}(\overline{\mathcal{H}yp}_{2,1})$, for $i=1,2$, the statement follows. 
\end{proof}

Note that
\[
\delta_{2,w} = \delta_{0:2}\cdot \pi_1^*\left[ \overline{\mathcal{H}yp}_{2,1}\right]  = \delta_{0:2} \cdot (3 \omega_1 - \lambda - \delta_{1} ).
\]
Hence, we can write
\begin{eqnarray}
\label{Hyp22}
\left[\overline{\mathcal{H}yp}_{2,2}\right] & = & \pi_1^* (3 \omega_2 - \lambda -\delta_{1}) \cdot  \pi_2^* (3 \omega_1 - \lambda -\delta_{1})  
-\delta_{0:2} \cdot (3 \omega_1 - \lambda - \delta_{1} ) - \gamma_{1:\emptyset}\nonumber\\
&=& (3 \omega_2 - \lambda -\delta_{1}) \cdot (3 \omega_1 - \lambda -\delta_{1}) -\delta_{0:2} \cdot (3 \omega_1 - \lambda - \delta_{1} ) - \gamma_{1:\emptyset}.
\end{eqnarray}

\subsection{The extremality for \texorpdfstring{$n=2$}{n=2}} 
\label{extrHyp22}

By Theorem \ref{rec}, in order to show that $[\overline{\mathcal{H}yp}_{2,n}]$ is rigid and extremal in ${\rm REff}^n(\BM_{2,n})$ for $2\leq n \leq 6$, it is enough to show that $[\overline{\mathcal{H}yp}_{2,2}]$ is rigid and extremal in ${\rm Eff}^2(\BM_{2,2})$.

\begin{thm}
\label{g=n=2}
$[\overline{\mathcal{H}yp}_{2,2}]$ is rigid and extremal in ${\rm Eff}^2(\BM_{2,2})$. 
\end{thm}

\begin{proof}
Suppose that
\begin{eqnarray}
\label{22sumi}
\left[\overline{\mathcal{H}yp}_{2,2}\right] = \sum_i a_i [X_i],
\end{eqnarray}
where $a_i > 0$ and $X_i$ is an irreducible codimension-two effective cycle on $\BM_{2,2}$ with $[X_i]$ not proportional to $[\overline{\mathcal{H}yp}_{2,2}]$, for all $i$. Let $\pi_j\colon \BM_{2,2} \rightarrow \BM_{2,1}$ be the map forgetting the point $j$, for $j=1,2$, and $\pi \colon \BM_{2,2} \rightarrow \BM_{2}$ be the map forgetting both marked points. Applying $(\pi_j)_*$ to (\ref{22sumi}), we obtain
\begin{eqnarray}
\label{22sumipij}
5\left[\overline{\mathcal{H}yp}_{2,1}\right] = \sum_i a_i (\pi_1)_*[X_i].
\end{eqnarray}

Pick a locus $X_i$ appearing on the right side of (\ref{22sumi}).
If $(\pi_1)_*[X_i] = (\pi_2)_*[X_i] = 0$, then either $X_i$ is contained in the inverse image via $\pi$ of a codimension-two effective cycle on $\BM_2$, or a general point of $X_i$ contains a smooth rational component with two marked points and two singular points. 
Note that a codimension-two locus in $\BM_2$ is a curve, and the cone of effective curves in $\BM_2$ is known to be spanned by the two one-dimensional boundary strata.
We deduce that $[X_i]$ is in the cone generated by the boundary strata classes $\delta_{01|}$, $\delta_{01|1}$, $\delta_{01|2}$, $\delta_{01|12}$, $\delta_{00}$, $\delta_{11|}$, and $\gamma_{1:\emptyset}$.

Suppose $(\pi_1)_*[X_i]$ is nonzero. Since $[\overline{\mathcal{H}yp}_{2,1}]$ is rigid and extremal in ${\rm Eff}^1(\BM_{2,1})$, from (\ref{22sumipij}) we deduce that $X_i\subset \pi_1^{-1}(\overline{\mathcal{H}yp}_{2,1})$. Hence, $(\pi_2)_*[X_i]$ is also nonzero, and $X_i\subset \pi_1^{-1}(\overline{\mathcal{H}yp}_{2,1}) \cap \pi_2^{-1}(\overline{\mathcal{H}yp}_{2,1})$. From Lemma \ref{pi1*pi2} we conclude that $[X_i]$ is supported on the locus of curves with a rational tail containing both marked points attached at a Weierstrass point of a genus-two curve, hence $[X_i]$ is a positive multiple of~$\delta_{2,w}$.

From (\ref{Hyp22dbs}), the class of $\overline{\mathcal{H}yp}_{2,2}$ lies outside the cone generated by $\delta_{2,w}$, $\delta_{01|}$, $\delta_{01|1}$, $\delta_{01|2}$, $\delta_{01|12}$, $\delta_{00}$, $\delta_{11|}$, and $\gamma_{1:\emptyset}$. Indeed, the coefficient of $\delta_{01|1}+\delta_{01|2}+\delta_{01|12}$ is negative. Hence $[\overline{\mathcal{H}yp}_{2,2}]$ is extremal in ${\rm Eff}^2(\BM_{2,2})$. 

The rigidity of $[\overline{\mathcal{H}yp}_{2,2}]$ follows from a similar argument. Suppose that $E:=m\, [\overline{\mathcal{H}yp}_{2,2}]$ is effective, for some positive $m$. Since $(\pi_j)_* (E)=5m\, [\overline{\mathcal{H}yp}_{2,1}]$ and $[\overline{\mathcal{H}yp}_{2,1}]$ is rigid, we have that $(\pi_j)_* (E)$ is supported on $\overline{\mathcal{H}yp}_{2,1}$, for $j=1,2$. Hence, the support of $E$ is in $\pi_1^{-1}(\overline{\mathcal{H}yp}_{2,1}) \cap \pi_2^{-1}(\overline{\mathcal{H}yp}_{2,1})$. From Lemma \ref{pi1*pi2}, $E$ is supported on the union of the loci $\overline{\mathcal{H}yp}_{2,2}$, $\Gamma_{1:\emptyset}$, and $\Delta_{2,w}$.  Since $E=m\, [\overline{\mathcal{H}yp}_{2,2}]$,  $E$ is supported only on $\overline{\mathcal{H}yp}_{2,2}$, and the statement follows.
\end{proof}

\section{The class of \texorpdfstring{$\overline{\mathcal{H}yp}_{2,3}$}{Hyp23}}
\label{Hyp23}

The aim of this section is to compute the class of $\overline{\mathcal{H}yp}_{2,3}$ in $A^3(\BM_{2,3})$. We first discuss a recursive relation between the classes of a partial closure of ${\mathcal{H}yp}_{2,n}$ and ${\mathcal{H}yp}_{2,n-1}$.

Recall the map $\pi_i\colon \BM_{g,n}\rightarrow \BM_{g,n-1}$ obtained by forgetting the $i$-th marked point, and the map $\rho_i\colon \BM_{g,n}\rightarrow \BM_{g,1}$ obtained by forgetting all but the $i$-th marked point.

\subsection{A recursive relation}
Let $\BM^{o}_{g,n}$ be the open locus in $\BM_{g,n}$ of stable curves with at most one non-disconnecting node. Let $\overline{\mathcal{H}yp}^{o}_{g,n}$ be the closure of ${\mathcal{H}yp}_{g,n}$ in $\BM^{o}_{g,n}$.
For $2\leq n \leq 6$, we note the following identity in $A^{n}(\BM^{o}_{2,n})$
\begin{align}
\label{Hyp2no}
\pi_n^*\left(\overline{\mathcal{H}yp}^{o}_{2,n-1}\right) \cdot \rho_n^*\left(\overline{\mathcal{H}yp}_{2,1}\right) \equiv \overline{\mathcal{H}yp}^{o}_{2,n} + 
\sum_{i=1}^{n-1} \pi_n^*\left( \overline{\mathcal{H}yp}^{o}_{2,n-1} \right)\cdot \delta_{0:\{i,n\}}.
\end{align}

Indeed, the intersection on the left-hand side consists of genus-two curves with a choice of $n$ ordered Weierstrass points, the first $n-1$ being distinct. The component $\overline{\mathcal{H}yp}^{o}_{2,n}$ corresponds to curves with all $n$ points distinct, and the component 
$\pi_n^*( \overline{\mathcal{H}yp}^{o}_{2,n-1} )\cdot \delta_{0:\{i,n\}}$ corresponds to curves with the $n$-th point coinciding with the $i$-th point, for $i=1,\dots,n-1$. A Weierstrass point on a smooth hyperelliptic curve of genus $g$ has weight $g(g-1)/2$. This explains the coefficient of $\overline{\mathcal{H}yp}^{o}_{2,n}$. Since the right-hand side is symmetric with respect to the first $n-1$ points, it is clear that all the components $\pi_n^*( \overline{\mathcal{H}yp}^{o}_{2,n-1})\cdot \delta_{0:\{i,n\}}$ have equal multiplicity, which, forgetting the point $n$, must equal $1$.

Using (\ref{Hyp2no}), one can recursively express the class of $\overline{\mathcal{H}yp}^{o}_{2,n}$ in terms of products of divisor classes. In the following, we derive a complete expression for the class of $\overline{\mathcal{H}yp}_{2,3}$ in $A^3(\BM_{2,3})$.

\subsection{A set-theoretic description}
To extend (\ref{Hyp2no}) with $n=3$ over $\BM_{2,3}$, we need to consider loci of curves with at least two non-disconnecting nodes.
Let $\Xi_i$ be the closure of the locus of curves with an elliptic component $[E, p_i, x,y]$ such that $2p_i\sim x+y$, and a rational component containing the other two marked points $p_j, p_k$, and meeting $E$ at the points $x,y$.
Let $\Theta$ be the closure of the locus of curves whose general element has a rational component $[R, p_1, p_2,  p_3, x, y]$ such that $2p_1 \sim 2p_2 \sim x+y$, and an elliptic component meeting $R$ at the points $x,y$.

\begin{prop}
\label{intHyp23}
We have
\begin{align*}
\pi_3^* \left(\overline{\mathcal{H}yp}_{2,2}\right) \cap \rho_3^* \left(\overline{\mathcal{H}yp}_{2,1}\right)  =  \overline{\mathcal{H}yp}_{2,3} \cup  
\pi_3^*\left( \overline{\mathcal{H}yp}_{2,2} \right)_{\big| \Delta_{0:\{1,3\}}} \cup  \pi_3^*\left( \overline{\mathcal{H}yp}_{2,2} \right)_{\big| \Delta_{0:\{2,3\}}}
\cup \Xi_1 \cup \Xi_2 \cup \Theta. 
\end{align*}
\end{prop}

\begin{proof}
The intersection 
\begin{eqnarray*}
\pi_3^* \left(\overline{\mathcal{H}yp}_{2,2}\right) \cap \rho_3^* \left(\overline{\mathcal{H}yp}_{2,1}\right)
\end{eqnarray*}
consists of stable curves $[C,p_1, p_2, p_3]$ with three marked Weierstrass points and with $p_1$ and $p_2$ corresponding to two different Weierstrass points. If the three points correspond to three different Weierstrass points, then $[C,p_1, p_2, p_3]$ is in $\overline{\mathcal{H}yp}_{2,3}$. If $p_3$ and $p_1$ correspond to the same Weierstrass point, then $[C,p_1, p_2, p_3]$ is in the restriction of $\pi_3^*\left( \overline{\mathcal{H}yp}_{2,2} \right)$ to $\Delta_{0:\{1,3\}}$. The case when $p_3$ and $p_2$ correspond to the same Weierstrass point is similar. Finally, restricting the intersection to the codimension-two boundary strata and using admissible covers to describe Weierstrass points on singular curves, we deduce that $\Xi_1$, $\Xi_2$, and $\Theta$ are the only additional components contained in the intersection, hence the statement.
\end{proof}

\subsection{The multiplicities}
Since the left-hand side of the expression in Proposition \ref{intHyp23} is symmetric with respect to the first two marked points, we conclude that the contributions of $\pi_3^*\left[ \overline{\mathcal{H}yp}_{2,2} \right]\cdot \delta_{0:\{1,3\}}$ and $\pi_3^*\left[ \overline{\mathcal{H}yp}_{2,2} \right]\cdot{ \delta_{0:\{2,3\}}}$ on the right-hand side are equal. Similarly for $\Xi_1$ and $\Xi_2$.
Hence we have
\begin{multline}
\label{abgd}
\pi_3^* \left[\overline{\mathcal{H}yp}_{2,2}\right] \cdot \rho_3^* \left[\overline{\mathcal{H}yp}_{2,1}\right]  = \\
 \alpha \left[\overline{\mathcal{H}yp}_{2,3}\right] 
 + \beta  \left( \delta_{0:\{1,3\}} + \delta_{0:\{2,3\}} \right) \cdot \pi_3^*\left[ \overline{\mathcal{H}yp}_{2,2} \right] 
 + \gamma (\left[ \Xi_1\right] + \left[ \Xi_2\right]) + \delta \left[ \Theta \right]
\end{multline}
for some coefficients $\alpha, \beta, \gamma, \delta$.

Forgetting the first marked point in (\ref{abgd}), the left-hand side is
\[
5 \left( \pi_1^*\left[ \overline{\mathcal{H}yp}_{2,1} \right] \cdot \pi_2^*\left[ \overline{\mathcal{H}yp}_{2,1} \right] \right) = 5 \left(\left[ \overline{\mathcal{H}yp}_{2,2} \right] + \gamma_{1:\emptyset} + \delta_{2,w} \right)
\]
by Lemma \ref{pi1*pi2},
hence we have
\[
5 \left(\left[ \overline{\mathcal{H}yp}_{2,2} \right] + \gamma_{1:\emptyset} + \delta_{2,w} \right) = (4\alpha+\beta) \left[ \overline{\mathcal{H}yp}_{2,2} \right] + (4\gamma + \delta) \gamma_{1:\emptyset} + 5\beta \, \delta_{2,w}.
\]
We deduce
\begin{align*}
\alpha & = \beta = 1, &  4\gamma + \delta &=5.
\end{align*}

In order to compute $\gamma$ and $\delta$, we consider the restriction of (\ref{abgd}) to the following three-dimensional test family. Attach at two points of an elliptic curve $E$ a rational tail containing the points marked by $2$ and $3$, and an elliptic tail $F$ containing the point marked by $1$. Consider the family obtained by varying $E$ in a pencil of degree $12$, by varying the point of intersection with the rational tail on the central elliptic component in which it lies, and by varying the point marked by $1$ on $F$ (see Figure \ref{fig:figu}). The base of this family is $Y\times F$, where $Y$ is the blow-up of $\mathbb{P}^2$ at the nine points of intersection of two general cubics.

 \begin{figure}[htbp]
 \centering
 \def\svgwidth{0.5\columnwidth}
 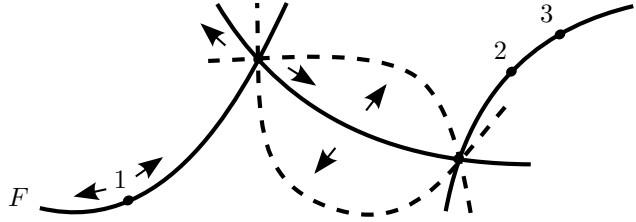
 \caption{How the general element of the family moves.}
 \label{fig:figu}
 \end{figure}

Let $H$ be the pull-back of the hyperplane class in $\mathbb{P}^2$, let $\Sigma$ be the sum of the nine exceptional divisors, and let $E_0$ be one of them. Denote by $\pi \colon Y\times F\rightarrow F$ the natural projection, and let $q=E\cap F$ be the singular point on $F$. The divisor classes on $\BM_{2,3}$ restrict as follows
\begin{eqnarray*}
\psi_1 &=& \pi ^* (q),\\
\delta_{\rm irr}  &=& 36 H -12 \Sigma \quad = \quad 12 \lambda,\\
\delta_{0:\{2,3\}} &=& {}-3H + \Sigma - E_0,\\
\delta_{1: \{1\}} &=& {}-3H + \Sigma - E_0 -\pi ^*(q),\\
\delta_{1:0} &=& E_0 + \pi ^*(q).
\end{eqnarray*}
From (\ref{Hyp21}) and (\ref{Hyp22old}), it follows that
\begin{multline*}
\rho_3^* \left[\overline{\mathcal{H}yp}_{2,1}\right] \cdot \pi_3^* \left[\overline{\mathcal{H}yp}_{2,2}\right]  = -\left( 3 \delta_{0:\{2,3\}} +\frac{1}{10} \delta_{\rm irr} + \frac{6}{5} (\delta_{1:\{1\}} + \delta_{1:0})  \right) \\
{}  \cdot \left(-6\psi_1 \cdot \delta_{0:\{2,3\}} -\frac{3}{2} \delta_{0:\{2,3\}}^2 -(\psi_1-\delta_{0:\{2,3\}})\cdot\left( \frac{21}{10}\delta_{1:\{1\}} + \frac{3}{5} \delta_{1:0} + \frac{1}{20} \delta_{\rm irr} \right) \right)
= 27, 
\end{multline*}
and similarly
\begin{align*}
\pi_3^*\left[ \overline{\mathcal{H}yp}_{2,2} \right]\cdot \delta_{0:\{2,3\}}   = -9.
\end{align*}
Note that this family meets $\Xi_1$ when $E$ degenerates to one of the $12$ rational nodal fibers, the rational tail is attached at a point colliding with the non-disconnecting node, and the point marked by $1$ differs from $q$ in $F$ by a nontrivial torsion point of order $2$ in ${\rm Pic}^0(F)$. The intersection is transverse at each point, hence we have
\[
\Xi_1=12 \cdot 3.
\]
All other classes in (\ref{abgd}) are disjoint from this family. We deduce the following relation
\[
27 = {}-9 \beta + 36 \gamma,
\]
hence we conclude that $\alpha = \beta = \gamma = \delta = 1$. We have thus proved the following statement.

\begin{prop}
\label{Hyp23-1}
One has
\[
  \left[\overline{\mathcal{H}yp}_{2,3}\right] = 
\pi_3^* \left[\overline{\mathcal{H}yp}_{2,2}\right] \cdot \left( \rho_3^* \left[\overline{\mathcal{H}yp}_{2,1}\right] 
 - \delta_{0:\{1,3\}} - \delta_{0:\{2,3\}}  \right)
 -  \left[ \Xi_1\right] - \left[ \Xi_2\right] -  \left[ \Theta \right].
\]
\end{prop}

\subsection{The boundary components}
It remains to compute the classes of $\Xi_1$, $\Xi_2$, and $\Theta$. Recall the classes $\gamma_{1:J}$ defined in the introduction.

\begin{lem}
\label{xitheta}
The following equalities hold in $A^3(\BM_{2,3})$
\begin{eqnarray*}
\left[ \Xi_i \right] &=& \left( 2\psi_i - \delta_{1:\{i\}} \right)\cdot \gamma_{1:\{i\}}   \quad \quad \mbox{for $i=1,2$,} \\
\left[ \Theta \right] &=& \gamma_{1:\emptyset}\cdot (\psi_1-\delta_{0:\{1,3\}})= \gamma_{1:\emptyset}\cdot (\psi_2-\delta_{0:\{2,3\}}). 
\end{eqnarray*}
\end{lem}

\begin{proof}
Consider the divisor $D_i$ of curves $[E,p_i,x,y]$ in $\MM_{1,3}$ such that $2p_i\sim x+y$. From \cite[\S 2.6]{MR1765541} or \cite[Proposition 3.1]{MR3231020}, one has $[\overline{D}_i] =2\psi_i-\delta_{1:\{i\}}$ in ${\rm Pic}(\BM_{1,3})$.
The locus $\Xi_i$ is the push-forward of $\overline{D}_i \times \BM_{0,4}\subset \BM_{1,3}\times \BM_{0,4}$ via the natural map $\BM_{1,3}\times \BM_{0,4}\rightarrow \Gamma_{1:\{i\}}\subset \BM_{2,3}$. 

Similarly, consider the map $\xi\colon \BM_{1,2}\times \BM_{0,5}\rightarrow\Gamma_{1:\emptyset}\subset\BM_{2,3}$ defined as 
\[
([E,x_1,y_1],[R,p_1,p_2,p_3,x_2,y_2]) \mapsto [E\cup_{x_1\sim x_2, y_1\sim y_2} R, p_1,p_2,p_3].
\]
The locus $\Theta$ is the push-forward via $\xi$ of the locus $\BM_{1,2}\times \pi_3^*({\rm point})\subset \BM_{1,2}\times \BM_{0,5}$, where $\pi_3\colon \BM_{0,5}\rightarrow \BM_{0,4}$ is the map obtained by forgetting the point $p_3$. The statement follows.
\end{proof}

\noindent From (\ref{Hyp21}), (\ref{Hyp22}), and Lemma \ref{xitheta}, 
the resulting expression in Proposition \ref{Hyp23-1} gives the statement in Theorem~\ref{hyp23class}.

\bibliographystyle{alpha}
\bibliography{Biblio.bib}

\end{document}